\documentclass{article} 
\usepackage{amsmath,amsfonts,amssymb, amsthm}
\usepackage{authblk}

\usepackage{centernot}
\usepackage{url}

\def\congruent{\equiv}

\def\intZ{ \mathbb{Z} }
\def\ratQ{ \mathbb{Q} }

\def\notdiv{\nmid}

\numberwithin{equation}{section}

\newtheorem{Theorem}{Theorem}

\newtheorem{Corollary}{Corollary}
\newtheorem{lemma}{Lemma}

\newtheorem{Observation}{Observation}
\newtheorem*{Observation*}{Observation A}

\def\Legendre#1#2{\left( {#1 \over #2} \right)}
\def\German{\mathfrak}

\newcommand*\conj[1]{\overline{#1}}

\def\bft{{\bf t}}
\def\bfs{{\bf s}}

\begin{document}

\title{Two terms with known prime divisors adding to a power}

\author[1]{Reese Scott}
\affil[1]{Somerville, MA, USA}
\author[2]{Robert Styer}
\affil[2]{Villanova University, Villanova, PA}



\maketitle

Let $c$ be a positive odd integer and $R$ a set of $n$ primes coprime with $c$. We
consider equations $X + Y = c^z$ in three integer unknowns $X$, $Y$, $z$, where
$z > 0$, $Y > X > 0$, and the primes dividing $XY$ are precisely those in
$R$. We consider $N$, the number of solutions of such an equation. Given a
solution $(X, Y, z)$, let $D$ be the least positive integer such that $(XY/D)^{1/2}$
is an integer. Further, let $\omega$ be the number of distinct primes dividing $c$.
Standard elementary approaches use an upper bound of $2^n$ for the number of
possible $D$, and an upper bound of $2^{\omega-1}$ for the number of 
ideal factorizations of $c$ in the field $\ratQ(\sqrt{-D})$ which can correspond (in a standard designated way) to a solution in which $(XY/D)^{1/2} \in \intZ$, and obtain $N \le 2^{n+\omega-1}$.  Here we improve this by finding an inverse proportionality relationship between a bound on the number of $D$ which can occur in solutions and a bound (independent of $D$) on the number of ideal factorizations of $c$ which can correspond to solutions for a given $D$.  We obtain $N \le 2^{n-1}+1$.  This bound is precise for $n<4$: there are several cases with exactly $2^{n-1} + 1$ solutions.  For higher values of $n$ the bound becomes unrealistic, but is nevertheless an improvement on bounds obtained previously by both elementary and non-elementary methods.  

\section{Introduction}   

In this paper we derive an upper bound on $N$, the number of solutions in integers $(X,Y,z)$ with $Y>X>0$ and $z>0$   to the equation 
$$ X+Y=c^z, $$
where $c$ is a fixed positive odd integer, $\gcd(XY,c) = 1$, and the set of primes in the factorization of $XY$ is prechosen.  Previous work on this problem includes both strictly elementary treatments and deeper, non-elementary approaches. 

The most common type of non-elementary approach uses results on S-unit equations to obtain a bound which is exponential in $s$, where $s$ is the number of primes dividing $XYc$.  A familiar general result of Evertse \cite{E} shows that there are at most $\exp(4s+6)$ solutions to the equation $x+y=z$ in coprime positive integers $(x,y,z)$ each composed of primes from a given set of $s$ primes.  It follows from a result of Beukers and Schlickewei \cite{BS} that $X + Y = c^z$ has at most $2^{16n+16}$ solutions where $n$ is the number of primes dividing $XY$.    

Treatments using strictly elementary methods take advantage of the fact that $c^z$ is a perfect power, and thus are often sharper than those obtained by more general non-elementary methods, especially when $c$ is divisible by few primes.  These elementary bounds are dependent not only on $n$, where $n$ is the number of distinct primes dividing $XY$, but also on $\omega$, where $\omega$ is the number of distinct primes dividing $c$.  Examples of such results are found in \cite{Le}, \cite{Cao}, and \cite{CCS}.  In this paper we show that strictly elementary methods can be used to obtain a bound which is independent of $\omega$ (note that the bound of $2^{16n+16}$ derived from the non-elementary result of Beukers and Schlickewei \cite{BS} is also independent of $\omega$).    
We will prove

\begin{Theorem}  
Let $n \ge 1$ and let $c$, $d_1$, \dots, $d_n$ be integers greater than 1 such that $c$ is odd and $d_1$, \dots, $d_n$ are coprime with $c$.  
Let $N$ be the number of solutions $(X, Y, z)$ to the equation 
$$ X + Y = c^z,   \eqno{(1.1)} $$
where $z \in \intZ$, $X = \prod_{i=1}^n d_i^{x_i}$, $Y = \prod_{i=1}^n d_i^{y_i}$, $x_i, y_i \in \intZ$, $\max(x_i, y_i) > 0$, $\min(x_i, y_i) = 0$, and $X<Y$.  

Then $N \le 2^{n-1} + 1$.

If the set $\{ \log(d_1), \log(d_2), \allowbreak \dots, \allowbreak  \log(d_n) \}$ is linearly independent over $\intZ$, then, letting $N_1$ be the number of solutions $(x_1, x_2, \dots, x_n, y_1, y_2, \dots, y_n, z)$ to (1.1), we have $N_1 \le 2^{n-1} + 1$.  
\end{Theorem}

The bound in Theorem 1, although not realistic for higher values of $n$, nevertheless improves both the elementary and non-elementary bounds mentioned above.  When $n<4$ the bound in Theorem 1 is precise: there are several cases with exactly $2^{n-1}+1$ solutions.  

To explain the key method which is new here, we briefly review the most common standard elementary approach to this problem: to simplify this explanation, in this paragraph and the next assume the $d_i$ in Theorem 1 are all prime; let $D$ be the least positive integer such that $\left( \frac{XY}{D} \right)^{1/2} \in \intZ$, square both sides of (1.1) and factor into ideals in the quadratic field $\ratQ(\sqrt{-D})$ to obtain $[X - Y + 2 \sqrt{-XY}] = \German{c}^{2z}$, $[X - Y - 2 \sqrt{-XY}] = \conj{ \German{c}}^{2z}$, where the ideal $\German{c}$ has norm $[c]$ and is not divisible by a principal ideal with a rational integer generator greater than one.  For each choice of $D$, there are $2^{\omega-1}$ possible pairs $\{  \German{c}, \conj{\German{c}} \}$.  For each such $\{  \German{c}, \conj{\German{c}} \}$ there is (with two exceptions) at most one solution to (1.1) (this is essentially an old result which we give as Lemma 1 in Section 2).  Roughly speaking, the standard elementary approaches obtain a bound on $N$ by multiplying the total number of possible $D$ by the total number of pairs $\{ \German{c}, \conj{\German{c}} \}$ which can occur for a given $D$.  This gives a bound of $2^{n + \omega -1}$, if one excludes from consideration the two exceptions mentioned above (see Lemma 1 in Section 2 for the two exceptions).  

In this paper we consider the congruence 
$$ X + Y \equiv 0 \bmod c \eqno{(1.2)} $$
noting that, just as in the treatment of (1.1) in the previous paragraph, each solution to (1.2) corresponds to a given $D$ and a pair of ideals $\{ \German{c}, \conj{\German{c}} \}$ in $\ratQ(\sqrt{-D})$.  Using a generalization of the methods of \cite{ScSt7}, we show that the larger the number of $D$ corresponding to solutions of (1.2) the smaller the number of pairs $\{ \German{c}, \conj{\German{c}} \}$ which can occur for a given $D$. More precisely, we obtain a bound $q$ on the number of pairs $\{ \German{c}, \conj{\German{c}} \}$ which can occur with a given $D$, and then, letting $p$ be the number of $D$ corresponding to solutions of (1.2), we show that $pq = 2^{n-1}$.   The bound $q$ is independent of the specific value of $D$.  

In \cite{ScSt7} this idea was used in the case $n=2$ to show that there are at most two solutions in positive integers $(x,y,z)$ to the equation $a^x + b^y = c^z$ where $a>1$, $b>1$, $ 2 \notdiv c$, improving the bound of $2^{\omega+1}$ in \cite{Le} and also improving the absolute bound of $2^{36}$ obtained by Hirata-Kohno \cite{H-K} using the non-elementary work of Beukers and Schlickewei cited above \cite{BS} (there are an infinite number of $(a,b,c)$ giving exactly two solutions).   

Our treatment in Theorem 1 is slightly more general than the usual treatment in that the $d_i$ are not necessarily prime, but this will not affect the theorem or its proof (see the parenthetical comment at the end of Section 2).  

From Theorem 1 we derive the following corollary which improves a result in \cite{CCS} in which the bound depends on $\omega$.

\begin{Corollary} 
Let $r$ and $s$ be positive integers, let $a$ and $b$ be integers greater than 1, and let $c$ be any odd positive integer prime to $ra$.  Then there are at most 4 solutions in positive integers $(x,y,z)$ to the equation 
$$ r a^x + s b^y = c^z \eqno{(1.3)}$$
except when (1.3) has a solution in which $\{ r a^x, s b^y \} = \{ 3 \left(\frac{3^{\nu-1} - 1}{8} \right),  \frac{3^{\nu+1} - 1}{8} \} $ with one of $a$, $b$ equal to 3 and $\nu>1$ an odd integer, in which case there is at most one further solution.  
\end{Corollary}

The following result, also an immediate consequence of Theorem 1, can be slightly improved using additional methods (although such improvements still lead to unrealistically large bounds when $w$ is large); but here we are concerned only with pointing out what Theorem 1 directly implies.

\begin{Corollary} 
Let $R$ be a finite set of primes with cardinality $w$ and let $W$ be the infinite set of all positive integers not divisible by any primes not in $R$.  Let $c$ be any positive odd integer none of whose prime divisors is in $R$.  Then there are at most $3^{w-1} + 2^{w-1}$ solutions $(A,B,z)$ to the equation 
$$ A + B = c^z, \eqno{(1.4)}$$
where $AB \in W$, $A < B$, and $z$ is a positive integer.  
\end{Corollary}

When $n=2$ we can improve Theorem 1:  

\begin{Theorem} 
In the notation of Theorem 1, if $n=2$ then $N \le 2$, except for the following choices of $(d_1, d_2,c)$, taking $d_1 > d_2$:  
$(3,2,5)$, $(5,2,3)$, $(2^{g-1} - 1, 2, 2^{g} - 1)$, $g>2$.  Each of these cases gives exactly three solutions as follows:
\begin{align*}
(d_1, d_2, c) &= (3,2,5): \\
& 3+2 = 5, 3 \cdot 2^3 + 1 = 5^2, 3^2 + 2^4 = 5^2.\\
(d_1, d_2, c) &= (5,2,3):\\
&  5 + 2^2 = 3^2, 5 \cdot 2^4 + 1 = 3^4, 5^2 + 2 = 3^3.\\
(d_1, d_2, c) &= (2^{g-1} - 1, 2, 2^{g} - 1):\\
&  (2^{g-1} -1) + 2^{g-1} = 2^g-1, \\
&  2^{g+1}(2^{g-1} - 1) +1 = (2^g-1)^2, \\
&  2(2^{g-1}-1) +1 = 2^g-1. 
\end{align*}
\end{Theorem}

Sections 2 and 3 will give a proof of Theorem 1.  The proofs of the Corollaries will be given in Section 4.   In Section 5 we will improve the bound on $N$ for $n \le 2$; in that section we will prove Theorem 2.

\section{Two Lemmas}  

For Lemmas 1 and 2 below, we need some definitions and some notation.  

Let $c$ be any positive odd integer, and let $D$ be a positive squarefree integer prime to $c$ such that $\Legendre{-D}{p} = 1$ for every prime $p$ dividing $c$.  Write $[\alpha]$ to represent the ideal generated by the integer $\alpha$ in the ring of integers of $\ratQ(\sqrt{-D})$.  

Let $C_D$  be the set of all pairs $\{ \German{c}, \conj{\German{c}} \}$ such that $\German{c}$, $\conj{\German{c}}$ are ideals in the ring of integers of $\ratQ(\sqrt{-D})$ such that $\German{c} \conj{\German{c}} = [c]$ and $\German{c}$ is not divisible by a principal ideal with a rational integer generator greater than one.  

For every such $\German{c}$ there is a unique integer $L$ such that $-(c-1)/2 \le L \le (c-1)/2$ and $\German{c} \mid [L+\sqrt{-D}]$.  To see this, note that there are $2^\omega$ possible choices of $\German{c}$, and also $2^\omega$ integers $L_i$ such that $L_i^2 \equiv -D \bmod c$, $-(c-1)/2 \le L_i \le (c-1)/2$, $ 1 \le i \le \omega$.  Each $[L_i + \sqrt{-D}]$ is divisible by exactly one $\German{c}$ (let $\German{c} = \German{p}_1^{\alpha_1} \dots \German{p}_\omega^{\alpha_\omega}$ and let $\German{c}_0$ be any other ideal appearing in one of the pairs in $C_D$; suppose $[L_i + \sqrt{-D}]$ is divisible by both $\German{c}$ and $\German{c}_0$; $\German{c}_0$ is divisible by the conjugate of $\German{p}_j$ for some $j$ such that $1 \le j \le \omega$, so that $[L_i + \sqrt{-D}]$ is divisible by $[p_j]$, a contradiction).  And no $\German{c}$ divides more than one $[L_i + \sqrt{-D}]$ (if $\German{c}$ divides both $[L_i + \sqrt{-D}]$ and $[L_h + \sqrt{-D}]$ then $L_i - L_h \in \German{c}$ so that $\German{c} \mid [L_i - L_h]$, which is impossible since $|L_i -L_h| < c$).  For Lemma 2 below we will use the following:  

\begin{Observation} 
If $\German{c} \mid [a + b \sqrt{-D}]$, then $a \equiv b L \bmod c$. 
\end{Observation}

Now let $(A,B)$ be an ordered pair of coprime positive integers such that  
$$ A + B \equiv 0 \bmod c,   \eqno{(2.1)}$$
with $D$ the least positive integer such that $\left(\frac{AB}{D}\right)^{1/2} \in \intZ$.  
We define the integer $\gamma(A,B)$ in $\ratQ(\sqrt{-D})$:
$$ \gamma(A,B) = A - B + 2 \sqrt{-AB }. $$
$\gamma(A,B) = (\sqrt{A} + \sqrt{-B})^2$.  (Note that $\gamma(B,A)= (\sqrt{-A} + \sqrt{B})^2 = B - A + 2\sqrt{-AB} = -\conj{\gamma}(A,B)$.) $[\gamma(A,B)] $ must be divisible by one of the ideals $\German{c}$ or $\conj{\German{c}}$ in exactly one of the pairs $\{ \German{c}, \conj{\German{c}} \} \in C_D$.  For any pair of coprime positive integers $(A,B)$ satisfying (2.1) with $D$ the smallest positive integer such that $\left(\frac{AB}{D}\right)^{1/2} \in \intZ$, we say that $(A,B)$ is {\it associated} with $\{ \German{c}, \conj{\German{c}} \} \in C_D$ if $[\gamma(A,B)]$ is divisible by $\German{c}$ or $\conj{\German{c}}$.  

The following notation is used for both lemmas:  $R$ is a given finite set of primes coprime with the given positive odd integer $c$;  $D_1$, $D_2$, \dots, $D_w$ are the positive squarefree divisors, including 1, of the product of all the primes in $R$; and $K = C_{D_1} \cup C_{D_2} \dots \cup C_{D_w}$, where $C_{D_j}$ is defined as above with $D=D_j$, $1 \le j \le w$ (note that $C_{D_j}$ is empty if $\Legendre{-D_j}{p} = -1$ for some $p$ dividing $c$).  

In Lemma 1, the set $T$ consists of the positive integers divisible by every prime in $R$ and by no other primes.  For Equation (2.2) below in Lemma 1, we say that a solution $(A,B,z)$ to (2.2) is associated with the pair $\{ \German{c}, \conj{\German{c}} \} \in K$ if the pair $(A,B)$ is associated with $\{ \German{c}, \conj{\German{c}} \}$.  Lemma 1 is essentially an old result:  see \cite{C}, \cite{Le}, \cite{Sc} for earlier versions.

\begin{lemma}    
Let $\{ \German{c}, \conj{\German{c}} \} \in K$.  Then  the equation 
$$ A + B = c^z,  \eqno{(2.2)} $$
in positive integers $A$, $B$, $z$ with $AB \in T$ and $A<B$, 
has at most one solution $(A,B,z)$ associated with $\{ \German{c}, \conj{\German{c}} \} $, except in the following two mutually exclusive cases:  

Case 1:  (2.2) has a solution associated with $\{ \German{c}, \conj{\German{c}} \}$ such that $(A, B)$ equals $\left( 3 \left(\frac{3^{\nu-1} - 1}{8} \right),  \frac{3^{\nu+1} - 1}{8} \right)$, where $\nu>1$ is an odd integer.      

Case 2:  (2.2) has a solution associated with $\{ \German{c}, \conj{\German{c}} \}$ with $B - A = 1$.

In both cases, $\{ \German{c}, \conj{\German{c}} \}$ has exactly two solutions of (2.2) associated with it, and it is the only pair in $K$ with this property.  Letting these two solutions be $(A,B,z)$ and $(A', B',z')$, for Case 1 we have $(A', B', z') = (B, 3^{2\nu} A, 3z) = (\frac{3^{\nu+1} - 1}{8}, 3^{2\nu+1} \left(\frac{3^{\nu-1} - 1}{8} \right), 3z)$,  and for Case 2 we have $(A', B',z') = (1, 4AB, 2z)= (1, 4A(A+1), 2z)$.    
\end{lemma}

\begin{proof}
Let $(A, B, z)$ be  a solution to (2.2) so that $AB \in T$, $A<B$, and $(A,B)=1$.  Choose $D$ to be the least integer such that $\left(\frac{AB}{D}\right)^{1/2} \in \intZ$, with $(A,B,z)$ associated with $\{ \German{c}, \conj{\German{c}} \} \in C_D$.  Write 
$$\gamma = \gamma(A,B,z)= A - B + 2 \sqrt{-AB}.  \eqno{(2.3)}$$
Let $j$ be the least integer such that $\German{c}^{2j} = [u_{2j} + v_{2j} \sqrt{-D}]$ for some integer $u_{2j}$ and some positive integer $v_{2j}$ such that $v_{2j}^2 D$ is divisible by every prime in $R$, so that $2 \mid v_{2j}^2 D$.  Note that $u_{2j}$ and $v_{2j}$ are unique, even when $D=1$ or 3.  Write $u_{2jt} + v_{2jt} \sqrt{-D} = (u_{2j} + v_{2j} \sqrt{-D})^t$ for all $t \ge 1$.  Since $\gamma \conj{\gamma} = c^{2z}$, by Lemma 2 of \cite{Sc} $j \mid z$ for every solution to (2.2) associated with $\{ \German{c}, \conj{ \German{c}} \}$, so that for every solution $(A,B,z)$ to (2.2) associated with $\{ \German{c}, \conj{\German{c}} \}$ there exists a $t \ge 1$ such that 
$$ \pm ( A - B \pm 2 \sqrt{-AB} ) = u_{2jt} + v_{2jt} \sqrt{-D}.  \eqno{(2.4)} $$
Conversely, for every $t \ge 1$ such that $v_{2jt} D \in T$, (2.2) has a solution $( A, B, z) = (A_{jt}, B_{jt}, jt)$ where 
$$ A_{jt} = \frac{c^{jt} - |u_{2jt}|}{2}, B_{jt} = \frac{c^{jt} + |u_{2jt}|}{2}  \eqno{(2.5)} $$
and 
$$A_{jt} + B_{jt} = c^{jt}, \  A_{jt} - B_{jt}  =  -|u_{2jt}|, \  A_{jt} B_{jt} = \frac{v_{2jt}^2 D}{4},  \eqno{(2.6)} $$ 
noting that $A_{jt} B_{jt} \in T$ since $8 \mid v_{2jt}^2 D = c^{2jt} - u_{2jt}^2$.  From (2.6) we see that $(A_{jt}, B_{jt})$ is the only possible choice of $(A,B)$ satisfying (2.4).  
 
By Lemma 1 of \cite{Sc} we must have $v_{2j}^2 D \in T$, so that (2.2) has a solution $(A_j, B_j, j)$ with 
$$ A_{j} = \frac{c^{j} - |u_{2j}|}{2}, B_{j} = \frac{c^{j} + |u_{2j}|}{2},  \eqno{(2.7)} $$ 
so that (2.6) holds with $t=1$.   

Suppose another solution associated with $\{ \German{c}, \conj{\German{c}} \}$ exists, and suppose further that this solution has $z = jt$ for some odd $t>1$.  By Lemma 1 of \cite{Sc} $v_{2j} \mid v_{2jt}$.  Let 
$$\frac{v_{2jt}}{v_{2j}} = k,  \eqno{(2.8)}$$
so that, by the last equation in (2.6), $A_{jt} B_{jt} = k^2 A_j B_j$.  So 
$$ c^{jt} = A_{jt} + B_{jt} \le A_{jt} B_{jt} + 1 = k^2 A_j B_j + 1 < k^2  \left(\frac{A_j + B_j}{2} \right)^2 + 1 =  \frac{k^2 c^{2j}}{4} + 1, $$
so that 
$$ 4 c^{j (t-2)} < k^2 +\frac{4}{9} \eqno{(2.9)} $$
where the term $4/9$ follows from $c^{2j} \ge 9$.  By (2.9) $k^2>1$, so there exists some prime $p \in R$ such that $p \mid \frac{v_{2jt}}{v_{2j}}$.  By Lemma 3 of \cite{Sc}, $p \mid t$, hence $p$ must be odd.  $v_{2j} D \in T$ and $v_{2jt} D \in T$, so $v_{2jp} D \in T$, so that (2.2) has a solution $(A_{jp}, B_{jp}, jp)$ such that (2.6) holds with $t=p$.  Assume first that either $p>3$ or $9 \mid A_j B_j$.  By Lemma 3 of \cite{Sc}, 
$$ \vert \frac{v_{2jp}}{v_{2j}} \vert = p, \eqno{(2.10)} $$
and we have (2.8) and (2.9) with $|k| = t = p$, which is impossible.  

So we must have $p = 3$ with $9 \notdiv A_j B_j$.  We argue as above with $p=3$ but with (2.10) replaced by 
$$ \vert \frac{v_{6j}}{v_{2j}} \vert = 3^\nu $$
for some positive integer $\nu$.  Considering the expansion of $(u_{2j}+v_{2j} \sqrt{-D})^3$, we have $\pm 3^\nu = \frac{ v_{6j} }{v_{2j}} = 3 u_{2j}^2 - v_{2j}^2 D$.  But also $u_{2j}^2 + v_{2j}^2 D = c^{2j}$, so that $4 u_{2j}^2 - c^{2j} = \pm 3^\nu$.  Consideration modulo 8 shows that we must have the upper sign on the right side of this equation with $\nu$ odd.  Taking the difference of squares of the left side we find
$$ c^{j} = \frac{ 3^\nu - 1}{2}, |u_{2j}| = \frac{3^\nu +1}{4}. $$
Substituting for $c^j$ and $|u_{2j}|$ in (2.7), we derive a Case 1 solution.  

Suppose there is a third solution with odd $t>3$.  Note that $ 3 \nmid v_{2j}$ (we are assuming $9 \nmid A_j B_j$).  We apply Lemma 3 of \cite{Sc}: $i=3$ is the least value of $i$ such that $3 \mid v_{2ji}$ (with $3^\nu \parallel v_{6j}$); and $i=9$ is the least value of $i$ such that $3^{\nu+1} \mid v_{2ji}$ (with $3^{\nu+1} \parallel v_{18j}$).  We now use these facts to consider the possibility of a third solution with odd $t>3$.  

We have already shown that $\vert \frac{v_{2jt}}{v_{2j}} \vert$ is greater than 1 and divisible by no primes other than 3. So $3 \mid v_{2jt}$, so that, by the previous paragraph along with Lemma 2 of \cite{Sc}, $6j \mid 2jt$.  If $3 \parallel t$, then, since $3^\nu \parallel v_{6j}$, also (by Lemma 3 of \cite{Sc}) $3^\nu \parallel v_{2jt}$.  Since $\frac{v_{2jt}}{v_{6j}}$ can be divisible by no primes other than 3, we must have, for odd $t>3$ such that $3 \parallel t$,  
$$ \vert  \frac{v_{2jt}}{v_{6j}} \vert = 1, $$
which we can treat in the same way as (2.8) to obtain a contradiction as in (2.9).  So we must have $ 9 \mid t$.  Therefore we have $v_{6j} D \mid v_{18j} D \mid v_{2jt} D \in T$, which requires $v_{18j} D \in T$, so that $(A, B, z) = (A_{9j}, B_{9j}, z_{9j})$ is a solution to (2.2).  From the previous paragraph we obtain 
$$ \vert \frac{v_{18j} } {v_{6j}} \vert = 3 $$
and using the same argument used above (in which we derived (2.9) from (2.8)) we obtain 
$$ 4 c^{9j-6j} < 3^2 + \frac{4}{9}$$
which is impossible, showing that there can be no third solution with $t$ odd.
 
Now suppose a solution associated with $\{ \German{c}, \conj{\German{c}} \}$ exists with $z=jt$ for some even $t$.  Then $v_{4j} D \in T$, which requires $u_{2j} = \pm 1$, so that, by the second equation in (2.6) with $t=1$, we have a Case 2 solution $(A,B,z) = ( A_j, B_j, j)$ with a second solution $(1, 4 A_j B_j, 2j)$.  Given what we have already shown for the case $t$ odd, we find that the existence of a further solution associated with $\{ \German{c}, \conj{\German{c}} \}$ such that $2j \mid z > 2j$ would require $u_{4j} = \pm 1$, which is impossible since $| u_{4j}| = |u_{2j}^2 - v_{2j}^2 D| = v_{2j}^2 D - 1 > 1$.  

So we have shown that there is at most one solution with odd $t>1$ which is associated with $\{ \German{c}, \conj{\German{c}} \}$, and, if such a solution exists, $(A_j, B_j, j)$ is Case 1; and we have shown there is at most one solution with even $t$ which is associated with $\{ \German{c}, \conj{\German{c}} \}$, and, if such a solution exists, $(A_j, B_j, j)$ is Case 2.   

It remains to show there is at most one $\{ \German{c}, \conj{\German{c}} \} \in K$  which has a Case 1 (respectively Case 2) solution associated with it, and that Cases 1 and 2 are mutually exclusive (that is, if a given pair $\{ \German{c}, \conj{\German{c}} \}$ in $K$ has a Case 1 solution associated with it, no pair in $K$ has a Case 2 solution associated with it).  For Case 1 we have $3^{\nu} - 1 = 2c^j $, so that, by Lemmas 3.1 and 3.2 of \cite{ScSt6}, $c$ determines $\nu$ which determines $A$ and $B$, so that a unique $\{ \German{c}, \conj{\German{c}} \} \in K$ is determined.  
And if $(A,B,z)$ is a Case 2 solution associated with a given $\{ \German{c}, \conj{\German{c}} \} \in K$, then $1+ 4 A (A+1) = c^{2z}$ is a solution to (2.2) associated with the same $\{ \German{c}, \conj{\German{c}} \}$; by Lemmas 3.1 and 3.2 of \cite{ScSt6}, $c$ and $R$ determine $A$, and a unique $\{ \German{c}, \conj{\German{c}} \} \in K$ is determined. 

So it remains to show that Cases 1 and 2 are mutually exclusive.  Suppose (2.2) has a Case 1 solution $(A_1, B_1, z_1)$ and a Case 2 solution $(A_2, B_2, z_2)$.  Take $A_1 = 3 \left( \frac{3^{\nu-1}-1}{8} \right)$, $B_1 = \frac{3^{\nu+1}-1}{8}$ where $\nu$ is an odd positive integer.  Note that $c^{z_1} = \frac{3^\nu-1}{2}$ so that $A_1 = \frac{c^{z_1} - 1}{4}$, $B_1 = \frac{3 c^{z_1} + 1}{4}$.  We first treat the case $z_1$ odd.  Take $A_2 = \frac{c^{z_2} - 1}{2}$, $B_2 = \frac{ c^{z_2} + 1}{2}$.  Note that $\frac{c+1}{2} \mid A_2 B_2$.  Suppose $\frac{c+1}{2}$ is a power of 2.  Then $c \congruent 3 \bmod 4$ and, since $z_1$ is odd, $c^{z_1} \congruent 3 \bmod 4$, contradicting $c^{z_1} = \frac{3^\nu - 1}{2} \congruent 1 \bmod 4$.  So there exists an odd prime $p$ such that $p \mid \frac{c+1}{2}$, implying $p \mid A_2 B_2$.  Since the set of primes dividing $A_1 B_1$ is the same as the set of primes dividing $A_2 B_2$, we will obtain a contradiction for the case $z_1$ odd by showing $p \notdiv A_1 B_1$: $p \mid {c+1} \mid c^{z_1} + 1 $ hence $p \notdiv A_1 = \frac{(c^{z_1} + 1) -2}{4}$ and $p \nmid B_1 = \frac{ 3 (c^{z_1} + 1) - 2}{4}$.  So we must have $2 \mid z_1$.  Since $\nu > 1$ is odd, it follows from an old result of Ljunggren \cite{Lj} that the equation $\frac{3^\nu - 1}{2} = y^2$ has as its only solution $(y,\nu) = (11,5)$.  So we must have $c^{z_1} = 11^2$, $A_1 = 30$, $B_1 = 91$.  If a Case 2 solution exists we must have, for some positive integer $s$, $11^s = A_2 + B_2$ where $| A_2 - B_2 | = 1$, so that $11^{2s} = 1 + 4 A_2 B_2$, where the set of primes dividing $4 A_2 B_2$ is $\{2, 3, 5, 7, 13\}$.  Noting that $11^3-1 = 2\cdot 5 \cdot 7 \cdot 19$, we see that $7 \mid 11^{2s} - 1$ only if $19 \mid 11^{2s} -1$, giving a contradiction in the case $z_1$ even.  
\end{proof}

For Lemma 2 below we will need a few further definitions.  

As above, let $c>1$ be an odd integer and let $d_1$, $d_2$, \dots, $d_n$ be integers greater than one all prime to $c$.  In Lemma 2 which follows we will be considering solutions $(x_1, x_2, \dots, x_n, y_1, y_2, \dots, y_n)$ (or, equivalently, solutions $(A,B)$) to the congruence 
$$ d_1^{x_1} d_2^{x_2} \dots d_n^{x_n} + d_1^{y_1} d_2^{y_2} \dots d_n^{y_n} \equiv 0 \bmod c,  \eqno{(2.11)} $$
where, for $1 \le i \le n$, $\min(x_i, y_i) = 0$ and $\max(x_i, y_i) \ge 0$, taking $n \ge 1$, and 
$$A = d_1^{x_1} d_2^{x_2} \dots d_n^{x_n}, B = d_1^{y_1} d_2^{y_2} \dots d_n^{y_n},  (A,B) = 1.$$
Note the difference between (2.11) and (1.1): in (1.1) we had $\max(x_i, y_i) >0$; (2.11) allows $\max(x_i, y_i)=0$.  Also, whereas (1.1) requires $X<Y$, (2.11) does not require $d_1^{x_1} d_2^{x_2} \dots d_n^{x_n} < d_1^{y_1} d_2^{y_2} \dots d_n^{y_n}$.  

We say that two solutions $(x_{1,1}, \allowbreak x_{2,1}, \allowbreak \dots, x_{n,1},\allowbreak y_{1,1}, \allowbreak y_{2,1}, \allowbreak \dots, \allowbreak y_{n,1})$ and $\allowbreak (x_{1,2}, \allowbreak x_{2,2}, \allowbreak \dots, \allowbreak x_{n,2}, \allowbreak y_{1,2}, y_{2,2}, \allowbreak \dots, y_{n,2})$ to (2.11) are in the same {\it parity class} if, for each $i$, $\max(x_{i,1}, y_{i,1}) \equiv \max(x_{i,2}, y_{i,2}) \bmod 2$.  

Let $R$ be the set of distinct primes dividing $d_1 d_2 \dots d_n$, where $d_1$, \dots, $d_n$ are as in (2.11), and let $D_1$, $D_2$, \dots, $D_w$, $C_{D_1}$, $C_{D_2}$, \dots, $C_{D_w}$, and $K$ be defined as above Lemma 1 for this $R$.  We say that any solution $(A,B)$ to (2.11) (and hence the corresponding solution $(x_1, \dots, x_n, y_1, \dots, y_n)$ to (2.11)) is associated with the pair $\{ \German{c} , \conj{\German{c}} \}$ in the set $K$  when the pair $(A,B) $ is associated with the pair $\{ \German{c}, \conj{ \German{c}}  \}$.  

Let $M$ be the multiplicative group of residue classes $m \bmod c$ such that $m^2 \equiv 1 \bmod c$ and there are integers $s_1$, \dots, $s_n$ such that  $\prod_{i=1}^n d_i^{s_i} \equiv m \bmod c$.
Define $q$ to be one half times the cardinality of $M$.

\begin{lemma}  
For every parity class of solutions of (2.11) there is a subset $K'$ of $K$ of cardinality at most $q$ such that every solution in the parity class is associated with a pair in $K'$.    
\end{lemma}

\begin{proof}
Let $(A,B)$ be a solution to (2.11) with $A = d_1^{x_1} d_2^{x_2} \dots d_n^{x_n}$, $B = \allowbreak  d_1^{y_1} d_2^{y_2} \allowbreak  \dots \allowbreak  d_n^{y_n}$.  
Let $V$ be the set of all $d_i$ such that $\max(x_i, y_i)$ is odd, so that we can write 
$$ \prod_{d_i \in V}  d_i = J^2 D \eqno{(2.12)} $$
where $J$ is a positive integer and $D$ is the least positive integer such that 
$ (\frac{AB}{D})^{1/2} \in \intZ$.  Let $D_1$ be the product of all $d_i \in V$ such that $d_i \mid A$, and let $D_2$ be the product of all $d_i \in V$ such that $d_i \mid B$.  
Now we have, for some positive integers $G$ and $H$, $A = D_1 G^2$, $B = D_2 H^2$, where $D_1 D_2 = J^2 D$, so that, for some pair $(\German{c}, \conj{\German{c}}) \in K$, one of $\German{c}$ or  $\conj{\German{c}}$, say $\German{c}$, satisfies 
$$ \German{c} \mid [\gamma(A,B)] = [ D_1 G^2 - D_2 H^2 + 2 G H J \sqrt{-D} ]. $$
Since $D_1 G^2 \equiv -D_2 H^2 \bmod c$, we have 
$$ \German{c} \mid [ 2 D_1 G^2 + 2 G H J \sqrt{-D}],$$
so that, since $2 \nmid c$, by Observation 1 we have 
$$  D_1 G^2 \equiv G H J L \bmod c,  \eqno{(2.13)} $$
where $L$ is defined as in Observation 1.   

Let $S$ be the set of all integers $k$ such that there exist integers $s_1$, $s_2$, \dots, $s_n$ such that $k \equiv d_1^{s_1} d_2^{s_2} \dots  d_n^{s_n} \bmod c$.  $D_1 \in S$ and $D_2 \in S$, so that $G \in S$ and $H \in S$.  So, by (2.13), we have 
$$ J L \in S. \eqno{(2.14)}$$

Now assume that congruence (2.11) has a solution $(A_1, B_1)$ distinct from $(A,B)$ but in the same parity class, so that (2.12) holds for the solution $(A_1, B_1)$.  Let $\German{c}_1$ divide $[\gamma(A_1, B_1)]$ and let $L_1$ be defined as $L$ is defined in Observation 1 (with $\German{c}_1$ replacing $\German{c}$).  The same argument as above gives 
$$ J L_1 \in S \eqno{(2.15)}$$
Let $L_1 \equiv \delta L \bmod c$ where $-(c-1)/2 \le \delta \le (c-1)/2$.  $L^2 + D \equiv  L_1^2 + D \equiv \delta^2 L^2 + D \equiv 0 \bmod c$ so that $\delta^2 \equiv 1 \bmod c$.  Thus, recalling (2.14) and (2.15), we have
$$ \delta \in S, \delta^2 \equiv 1 \bmod c.  \eqno{(2.16)} $$ 

Since we assumed that (2.11) is solvable, we see that $-1 \bmod c \in M$ and so, if $m \in M$ then also $-m \in M$. 
If (2.11) has yet another solution $(A_2, B_2)$ in the same parity class as $(A,B)$ with $L_2 \equiv -\delta L \bmod c$, then the solutions $(A_1, B_1)$ and $(A_2, B_2)$  are associated with the same $\{ \German{c} , \conj{ \German{c} } \}$.  

Since $\delta$ is in one of the residue classes $m \in M$, we conclude that the pair $\{ \German{c}, \conj{\German{c}} \} \in K$ 
with which a solution in our parity class is associated is uniquely determined by a pair $ \{ m, -m \}$ 
with $m \in M$.  Since the number of such pairs is equal to $q$, this proves our lemma. 
\end{proof}

Let $p$ be the number of parity classes for 
$(\max(x_1, y_1), \allowbreak  \dots, \allowbreak  \max(x_n, y_n))$ 
which occur in solutions to (2.11).  Let $q$ be as in Lemma 2.  Combining Lemma 1 and Lemma 2, and letting $N$ be the number of solutions $(X, Y, z)$ to (1.1), we have 
$$ N \le p q +1, \eqno{(2.17)} $$
where the \lq $+1$' in (2.17) is needed only when Case 1 or Case 2 of Lemma 1 occurs.  

If the set $\{ \log(d_1), \log(d_2), \dots, \log(d_n) \}$ is linearly independent over $\intZ$, then, in (1.1), $X$ and $Y$ uniquely determine $x_i$, $y_i$ for all $1 \le i \le n$, so that, if $N_1$ is as in Theorem 1, then $N_1 \le pq + 1$.

(When the set $\{ \log(d_1), \log(d_2), \dots, \log(d_n) \}$ is not linearly independent over $\intZ$, or when one or more $d_i$ in Theorem 1 is a perfect square, then there may be more than one parity class corresponding to a given $D$, but this does not affect (2.17).)

\section{$pq = 2^{n-1}$}  

For Section 3 we define the congruence relation modulo $c$ to be extended to rational numbers with denominators coprime to $c$.  

If $(x_1, \dots, x_n, y_1, \dots, y_n)$ is a solution of (2.11) then $\bft = (t_1, \dots, t_n) = (x_1 - y_1, \dots, x_n - y_n)$ satisfies
$$ d_1^{t_1} \dots d_n^{t_n} \equiv -1 \bmod c.  \eqno{(3.1)} $$
Further, two solutions of (2.11) are in the same parity class if and only if the corresponding vectors $\bft$ with
(3.1) are congruent modulo 2. So if $W$ denotes the set of vectors $\bft \in \intZ^n$ satisfying (3.1), and 
$\varphi : \intZ^n \rightarrow \intZ^n/2\intZ^n$ is the group homomorphism sending $\bft$ to $\bft \bmod 2$, we see that the number $p$ of parity classes of solutions of (2.11) is $\#\varphi(W)$, the cardinality of $\varphi(W)$.  Assume that $W$ is nonempty, otherwise $p=0$ and we are done.  Take $\bft_0 \in W$.  Then $W = \bft_0 + U = \{\bft_0 + \bft : \bft \in U \}$, where $U$ is the subgroup of $\intZ^n$ given by 
$$ U = \{ \bft = (t_1, \dots, t_n) \in \intZ^n : d_1^{t_1} \dots d_n^{t_n} \equiv 1 \bmod c \}.  $$
So $\varphi(W) = \varphi(\bft_0) + \varphi(U)$, hence $\#\varphi(W) = \#\varphi(U)$.  This leads to 
$$ p = \#\varphi(U) = \#( U/ U\cap 2\intZ^n ). \eqno{(3.2)} $$
Next, let $q$ be defined as in the paragraph immediately preceding Lemma 2.  Let $U'$ be the subgroup $U' = \{ \bfs = (s_1, \dots, s_n) \in \intZ^n : 2\bfs \in U \}$.  Then 
$$ \varphi' : U' \rightarrow (\intZ/c\intZ)^*, \varphi'(\bfs) = d_1^{s_1} \dots d_n^{s_n} \bmod c  $$
is a group homomorphism with image 
$$ M = \{ \delta \in (\intZ/c\intZ)^* : \delta^2 \equiv 1 \bmod c, \exists \bfs \in \intZ^n {\rm\ with \ } \delta \equiv d_1^{s_1} \dots d_n^{s_n} \bmod c \}$$
and kernel $U$.  Hence $M \cong U'/U$.  So 
$$ q = \frac{1}{2} \#M = \frac{1}{2} \#(U'/U).  \eqno{(3.3)} $$
Combining (3.2) and (3.3) with the facts that $U \cap 2\intZ^n = 2U'$ and that $U'$ is a free abelian group of rank $n$, we arrive at 
$$ pq = \frac{1}{2} \#(U/2U') \cdot  \#(U'/U) = \frac{1}{2} \#(U'/2U') = 2^{n-1}.  $$
This completes the proof of Theorem 1.

\section{Proofs of Corollaries}  

For the proof of Corollary 1, we will use the following more general result, which is an immediate consequence of Theorem 1:

\begin{lemma} 
Let $c$ be an odd integer greater than one, let $r$ and $s$ be positive integers prime to $c$, and let $d_1$, $d_2$, \dots, $d_n$ be integers greater than one which are prime to $c$, $n \ge 1$.  Then there are at most $2^{n}+1$ solutions $(X, Y, z)$ to the equation 
$$ rX + sY = c^z,   \eqno{(4.1)} $$
where $z>0$, $X = \prod_{i=1}^n d_i^{x_i}$, $Y = \prod_{i=1}^n d_i^{y_i}$, $\max(x_i, y_i) > 0$, $\min(x_i, y_i) = 0$, and, when $rs=1$, $X<Y$. 
\end{lemma} 

\begin{proof} 
We proceed as in Sections 2 and 3 of the proof of Theorem 1, except that $\delta \equiv -1 \bmod c $ is no longer necessarily possible, so that we need to replace (3.3) by the equation $q=\#(M)=\#(U'/U)$, doubling the bound obtained in Section 3.  We then apply Lemma 1 to obtain the bound $2^{n}+1$.     
\end{proof}

\begin{proof}[Proof of Corollary 1]  
After Lemma 3 it suffices to point out that (1.3) cannot satisfy Case 2 of Lemma 1 since, as pointed out in the proof of Lemma 1, the presence of a Case 2 solution would require a further solution in which $\min(ra^x, sb^y) = 1$, which is impossible in (1.3).  
\end{proof}

Corollary 1 can also be proven without using the methods of Theorem 1 and Lemma 3, instead using the methods of \cite{ScSt7}: in this way, a result very similar to Corollary 1 is obtained in \cite{DYL}, which came to our attention after completion of this paper.   \cite{DYL} also gives a condition under which (1.3) has at most two solutions.   

\begin{proof}[Proof of Corollary 2]
We need to consider all cases of (1.1) in which $\{ d_1, d_2, \allowbreak  \dots, \allowbreak  d_n \}$ is a subset of $R$ such that this subset does not lead to an immediate contradiction modulo $c$.  For $k$ such that $0 \le k \le w$, the number of such subsets with cardinality $k$ is at most $\binom{w-1}{k-1}$, since we do not need to consider subsets which do not contain the prime 2, since $c$ is odd (note $k = 0$ is impossible).  Thus, letting $N_0$ be the number of solutions $(A,B,z)$ to (1.4), and letting $g = k-1$, we have 
$$ 
\begin{aligned}
N_0 &\le \sum_{k=1}^{w} \binom{w-1}{k-1}  \left( 2^{k-1} + 1 \right)       \\
  &= \sum_{g=0}^{w-1} \binom{w-1}{g} 2^g + \sum_{g=0}^{w-1} \binom{w-1}{g}  \\
  &= 3^{w-1} + 2^{w-1}.   \\
\end{aligned} 
$$
\end{proof}

To put Corollary 2 in the context of earlier results in this paper, let $p$ and $q$ be as in Section 2, and let $F= pq$.  Using the results of Section 2, it can be easily shown that there are $F$ infinite families of solutions to (2.11), each infinite family corresponding to a combination of a parity class with an ideal factorization, where here \lq parity class' refers to one of the $p$ parity classes allowing solutions to (2.11), and \lq ideal factorization' refers to one of the $q$ ideal factorizations with which a solution to (2.11) in a given parity class can be associated.     Note that $p=F=0$ if and only if $u_c(d_i)$ is odd for every $i$.  

The bound in Corollary 2 can equal the actual number of solutions only when, for every subset of $R$ containing 2, $p \ne 0$ and every one of the $F$ infinite families of solutions to (2.11) yields a solution to (1.1), with one such family yielding two solutions to (1.1).  This occurs when $w=1$, $R=\{ 2\}$, $c=3$ (which gives two solutions), and when $w=2$, $R=\{5,2\}$, $c=3$ (which gives five solutions).  That these are the only such cases follows from the results in the next section.

 \section{Sharper results for $n \le 2$}  

Let $N$ and $n$ be as in Theorem 1.  In this section we give improvements on the bound on $N$ for $n \le 2$.  When $n=1$, it is an elementary result that the only case with $N>1$ is $d_1 = 2$, $c=3$ (see Theorem 4 of \cite{Cas}, Theorem 6 of \cite{Lev}, or Lemmas 3.1 and 3.2 of \cite{ScSt6} for three different treatments).  

When $n=2$ we have Theorem 2, whose proof follows.

\begin{proof}[Proof of Theorem 2]    
Recalling that the \lq +1' in (2.17) is needed only when Case 1 or Case 2 of Lemma 1 occurs, we see that we can assume that (1.1) has a Case 1 or Case 2 solution when (1.1) has more than two solutions.  Take $d_1 > d_2$.  The only instance of Case 1 of Lemma 1 for which we can have $n=2$ is $(d_1, d_2, c) = (10, 3, 13)$, which has only two solutions (in the notation of Theorem 1, we have $2 \notdiv \max(x_1, y_1)$, by consideration modulo 13; since $3^2 + 2^5 \cdot 5 = 13^2$, there can be no solution with $2 \mid \max(x_2, y_2)$, by Lemma 1; thus there can be no third solution by Lemma 1).    So if (1.1)  with $n=2$ has three solutions, we must have Case 2 of Lemma 1 and (1.1) must have at least the two solutions $(X_1, Y_1, z_1)$ and $(X_2, Y_2, z_2)$ with   
$$ \vert X_1 - Y_1 \vert = 1, (X_2, Y_2) = ( 1, 4 X_1 Y_1), z_2 = 2 z_1.  \eqno{(5.1)} $$
Since we are considering (1.1) with $n=2$, from (5.1) we see that one of $d_1$, $d_2$ must be 2 or 4, so it suffices to let $d_2=2$.  

For convenience, in what follows, we replace the restriction $X<Y$ in (1.1) by the restriction $ 2 \mid Y$ (which serves the same function of ensuring that $(X,Y,z)$ is not considered a different solution than $(Y,X,z)$).  So we can take $( X_1, Y_1 ) = ( d_1^{x}, 2^{y} ) = ( 2^{y} + (-1)^\epsilon, 2^{y} )$ where $\epsilon \in \{0,1\}$.  If $\epsilon = 1$, then it is a familiar elementary result that $x = 1$, and we have the infinite family which is the final exception in the formulation of Theorem 2.  Note that members of this exceptional infinite family are the only cases in which (1.1) has more than one solution with $\min(X,Y)=1$, by Lemma 3.2 of \cite{ScSt6}.   

So it remains to consider the case $\epsilon = 0$.  Let $y = g-1$ so that $c^{z_1} = 2^{g}+1$.  If $d_1^{x} = 2^{g-1} + 1 = 9$ we can assume $d_1=3$ and take $(d_1, d_2, c) = (3, 2, 17)$; by Lemma 2 of \cite{ScSt1}, we see that this case allows only one solution with $\min(X,Y) > 1$, so that $(d_1, d_2, c) = (3,2,17)$ allows only two solutions.  If $c^{z_1} = 2^g + 1 = 9$, we have, in addition to the two solutions of (5.1), the solution $5^2 + 2 = 3^3$ and we obtain the second exception listed in the formulation of the theorem.   So now it is a familiar elementary result that we can take $x = z_1 = 1$.  

Write 
$$ X + Y = c^z \eqno{(5.2)} $$
where (5.2) is (1.1) with 
$$ n = 2, d_1 = d = 2^{g-1}+1, d_2 = 2, c = 2^{g}+1.  $$
(5.2) has the two solutions 
$$ 2^{g-1} + d = c \eqno{(5.3)} $$
and 
$$ 1 + 2^{g+1} d = c^2. \eqno{(5.4)} $$
We need to show that (5.3) and (5.4) are the only solutions to (5.2), not including the first exceptional case in the formulation of the theorem.  

$g$ is the least positive integer such that
$$ 2^g \equiv -1 \bmod c, \eqno{(5.5)} $$
so 
$$ u_c(2) = 2g  \eqno{(5.6)} $$
where $u_c(a)$ is the least positive integer $\mu$ such that $a^\mu \equiv 1 \bmod c$.  
We have 
$$d \equiv -2^{g-1} \equiv 2^{2g-1} \bmod c. \eqno{(5.7)}$$
Therefore, by (5.5) and (5.6), for any solution $d^{x}+2^{y} = c^{z}$, we have $2^{x(2g-1)} \equiv d^{x} \equiv -2^y \equiv 2^{y+g} \bmod {c}$, so that $x+y+g \equiv 0 \bmod 2g$.  Thus we have 
$$ 2 \mid g \implies 2 \mid x-y, 2 \nmid g \implies 2 \nmid x-y.  \eqno{(5.8)}$$
By (5.7), for any pair of integers $s_1$ and $s_2$ there exists a nonnegative integer $k < 2g$ such that 
$$ d^{s_1} 2^{s_2} \equiv 2^k \bmod c.  \eqno{(5.9)}  $$  
From (5.5) and (5.6) we see that $(2^k)^2 \equiv 1 \bmod c$ only if $g \mid k$, that is, only if $2^k \equiv \pm 1 \bmod c$.   So from (5.9) we see that, if $\delta \equiv d^{s_1} 2^{s_2} \bmod c$ satisfies (2.16), then $\delta \equiv \pm 1 \bmod c$.  So any two solutions to (5.2) in the same parity class must be associated with the same pair $\{ \German{c}, \conj{ \German{c}} \}$ in $K$, where $K$ is as in Section 2 (recall the proof of Lemma 2).   Thus, (5.3) and (5.4) are the only solutions to (5.2) such that $XY = d^{k_1} 2^{k_2}$ where $k_1$ is odd and $k_2 - g$ is odd.  Also, recall (5.2) has no further solutions with $\min(X,Y)=1$, by Lemma 3.2 of \cite{ScSt6}.  

Thus, if (5.2) has another solution $(X_3, Y_3, z_3)$, then, by (5.8), we can take  
$$ \{ X_3, Y_3 \} = \{ d^{x_3}, 2^{y_3} \} $$
where $x_3$ is even and $y_3 - g$ is even (here the subscripts are not as in Theorem 1 but rather are used simply to indicate that we are dealing with a third solution).  

Suppose $g$ is odd so that $y_3$ is odd.  Since $2^{g-1} \equiv -1 \bmod d$, we have $2^{y_3} \not\equiv \pm 1 \bmod d$.  But $c^{z_3} \equiv \pm 1 \bmod d$, a contradiction.  

So $g$ is even, so that we have $2 \mid x_3$, $2 \mid y_3$.  Since when $g$ is even $d \equiv 0 \bmod 3$ and $c \equiv 2 \bmod 3$, we must also have $2 \mid z_3$, so that $y_3 > 2$.  This gives rise to the Pythagorean triple $(d^{x_3/2}, 2^{y_3/2}, c^{z_3/2}) = (a^2-b^2, 2ab, a^2 + b^2)$ where $a = 2^{(y_3/2) - 1}$ and $b = 1$, so that $2^{y_3 - 2} -1 = d^{x_3 / 2} = 2^{g-1} + 1$ (noting $x_3/2 = 1$ since $d^{x_3/2} + 1$ is a power of 2), hence $d^{x_3/2}=d=3$, $g=2$, $c=5$, $x_3 = 2$, $y_3=4$, and we obtain the first exceptional case in the theorem.  
\end{proof}

There are many cases with $N=2$ when $n=2$:  there are at least five infinite families of such cases, and many anomalous cases which are not members of known infinite families (the anomalous case with the largest $c^z$ which we have found is $10^5 + 41^3 = 411^2$ which has the second solution $1 + 10 \cdot 41 = 411$).  It seems to be a difficult problem to estimate the nature and extent of such double solutions; if one excludes from consideration cases in which $\min(X,Y)=1$, then a conjecture on double solutions is given at the end of \cite{ScSt7}.

The contribution of the referee to this paper is significant: the lengthy referee's report deals with many details, and also provides a major improvement in Section 3, where we have replaced our original treatment with an essentially new and shorter treatment given by the referee.  We are certainly grateful for the referee's quick response, detailed attention, and important insight.

\bigskip 

\noindent
{\bf Appendix 1: $n=N=2$}   

\bigskip 

We investigate the case $n=N=2$, now allowing $c$ to be even or odd in (1.1).  In this appendix we take $d_1 < d_2$ (whereas in Theorem 2 we took $d_1 > d_2$).

A computer search reveals many cases of double solutions when $n=2$.  
Many of these cases are members of one of the following infinite families:
\begin{align*} 
(I) \quad (d_1, d_2, c) &= (k, \frac{k^m-1}{k-1}, \frac{k^{m+1}-1}{k-1}):\\
&  k^m+ \frac{k^m-1}{k-1}= \frac{k^{m+1}-1}{k-1}, k \frac{k^m-1}{k-1} +1 =\frac{k^{m+1}-1}{k-1}
\end{align*}
where $k \ge 2$ and $m \ge 2$ are integers.  For $k=2$, (I) gives the first and third solutions to the infinite family in the statement of Theorem 2.  All other choices of $k$ give exactly two solutions.

Next we have   
\begin{align*}
(II) \quad (d_1, d_2, c) &= (2, 2^r \pm 1, 2^{r+1} \pm 1):\\
& 2^r + (2^r \pm 1) = 2^{r+1} \pm 1, 2^{r+2}(2^r \pm 1) + 1 = (2^{r+1} \pm 1)^2
\end{align*}
for $r \ge 1$ where all the signs agree. Taking the lower sign with $r \ge 2$ gives the first and second solutions to the infinite family in the statement of Theorem 2.  Taking the upper sign with $r=1$   gives the first and second solutions to $(d_1, d_2, c) = (2,3,5)$ which is the first exceptional case in the statement of Theorem 2.  Taking the upper sign with $r=2$ gives the first and second solutions to the case $(d_1, d_2, c)=(2,5,3)$ which is the second exceptional case in the statement of Theorem 2. All other choices of $r$ and sign give exactly two solutions.  Taking the lower sign with $r=1$ gives the $n=1$ case mentioned at the beginning of Section 5.  When $r=2m$ is even, we can derive related infinite families with $d_1=4$, and these derived families have exactly two solutions. A further infinite family  can be derived using the following observation:

\begin{Observation*}  
If we have a double solution $(a,b,c): a^q b+ 1 = c^r, a^s + b = c^t$, we can obtain another double solution $(a,c,c^r+a^{s+q}): a^q c^t + 1 = c^r + a^{s+q}, a^{s+q} + c^r = c^r + a^{s+q}$.  
\end{Observation*}

Using Observation A, from (II) we obtain 
\begin{align*}
(III) \quad (d_1,\  d_2, c) &= (2, 2^r \pm 1, 2^{2r+1} \pm  2^{r+1}+1): \\
& 2^{2r} + (2^r \pm 1)^2 = 2^{2r+1} \pm 2^{r+1}+1, \\
&  2^{r+1}(2^r \pm 1) + 1= 2^{2r+1} \pm 2^{r+1}+1
\end{align*}  
for $r \ge 2$, noting that $r$ in (III) corresponds to $r+1$ in (II).  Taking $r=1$ with the upper sign in (III) gives a case not derived in the same way from (II).  Taking $r=1$ with the lower sign in (III) gives an $n=1$ case with just one solution.  Taking the lower sign with $r=2$ in (III) gives the second and third solutions of the first exceptional case in the statement of Theorem 2.  All other choices of $r$ and sign in (III) give exactly two solutions.  

The next infinite family has $(d_1, d_2, c) = (2^r-1, 2^{r}+1, 2)$ which can be shown to include all $(d_1, d_2, c)$ with $c=2$ for which (1.1) has more than one solution, except for the case $(d_1, d_2, c)=(3,13,2)$.  (This can be shown using Theorem 6 of \cite{Sc}, Lemma 3.2 of \cite{ScSt6}, and (8) of Theorem 1 of \cite{Sc}.  Note that (8) of Theorem 1 of \cite{Sc} does not include the case $(d_1, d_2, c)=(3,5,2)$, since this case gives four solutions ($3+5=2^3$, $3 \cdot 5 + 1 = 2^4$, $3^3 + 5 = 2^5$, $5^3 + 3 = 2^7$); in all other cases, the infinite family which follows gives exactly two solutions.)  
\begin{align*} 
(IV) \quad (d_1, d_2, c) &= (2^r-1, 2^{r}+1, 2): \\
& (2^r-1)+(2^r+1) = 2^{r+1}, (2^r-1)(2^r+1)+1 = 2^{2r}
\end{align*}
for $r \ge 2$.  Since in (IV) the exponent of each $d_i$ ($i \in \{ 1,2\}$) is the same in both solutions, we can apply Observation A in two different ways, obtaining a new derivation of (III).  

Next is the only known infinite family with $n = 2$ giving two solutions to (1.1) both of which have $X>1$: 
\begin{align*}
(V) \quad (d_1, d_2, c) &= (2, 2^r-1, 2^{r}+1):\\
& 2 + (2^r-1) = 2^{r} + 1, 2^{r+2}+(2^r-1)^2  = (2^{r}+1)^2
\end{align*} 
for $r \ge 2$.  All cases of (V) have exactly two solutions.  From (V) we also obtain a new case with $N=2$: 
$$ (d_1, d_2, c) = (2, 7, 3).$$

We now list anomalous double solutions which do not seem to belong to infinite families.

Several such anomalous examples are double solutions to Pillai's equation, that is, cases in which the exponent of one of $d_i$ ($i \in \{1,2\}$) is the same in both solutions; see the list of such double solutions in (1.2) of [M. Bennett, On Some Exponential Equations of
S. S. Pillai, {\it Canad. J. Math.}, {\bf 53}, (2001)].  
The Pillai double solution $3-2= 3^2-2^3 = 1$ corresponds to the case $(d_1, d_2, c) = (2,1,3)$ giving the $n=1$ case mentioned at the beginning of Section 5, and two other Pillai double solutions, $2^3-3=2^5-3^3=5$ and $2^3-5 = 2^7-5^3=3$, correspond to the exceptional third and fourth solutions to the case $(d_1, d_2, c) = (3,5,2)$ mentioned above.  The other three known Pillai solutions with relatively prime terms give these anomalous examples:
\begin{align*}
(d_1, d_2, c) &= (3,13,2): \\
& 3+13 = 2^4, 3^5 + 13=2^8.\\
(d_1, d_2, c) &= (2,89,91):\\
&  2+89=91, 2^{13}+89=91^2.\\
(d_1, d_2, c) &= (3,10,13):\\
&  3+10 = 13, 3^7 + 10=13^3.
\end{align*}

From any Pillai double solution we can obtain related solutions.  
From $(d_1, d_2, c) = (3,13,2):  3+13 = 2^4, 3^5 + 13=2^8$  we obtain
\begin{align*}
(d_1, d_2, c) &= (2,3,259):\\
&  2^8+3 = 2^4+3^5=259.
\end{align*}
From $(d_1, d_2, c) = (2,89,91):  2+89=91, 2^{13}+89=91^2$ we obtain 
\begin{align*}
(d_1, d_2, c) &= (2,91,8283):\\
&  2 + 91^2 = 2^{13} + 91=8283.
\end{align*}
From $(d_1, d_2, c) = (3,10,13):  3+10 = 13, 3^7 + 10=13^3$ we obtain 
\begin{align*}
(d_1, d_2, c) &= (3,13,2200):\\
&  3+13^3 = 3^7 + 13=2200.
\end{align*}
From $(d_1,d_2,c) = (2,1,3): 2+1=3, 2^3+1=3^2$ we obtain 
\begin{align*}
(d_1, d_2, c) &= (2,3,11):\\
&  2+3^2= 2^3+3=11.  
\end{align*}
From various combinations of three solutions of $(d_1, d_2, c) = (3,5,2): 3 + 5 = 2^3, 3^3 + 5 = 2^5, 3 + 5^3 = 2^7$, we obtain two further anomalous cases:  
\begin{align*}
(d_1, d_2, c) &= (2,3,35):\\
&  2^3+3^3= 2^5+3=35,\\
(d_1, d_2, c) &= (2,5,133):\\
&  2^3+5^3= 2^7+5=133.  
\end{align*}     
At this point we have listed all known cases with $n=2$ for which (1.1) has exactly two solutions both of which have $X>1$ (see the conjecture in Section 3 of \cite{ScSt7}).  Noting that Lemma 3.2 of \cite{ScSt6} shows that members of the exceptional infinite family in the statement of Theorem 2 are the only cases in which (1.1) has more than one solution with $X=1$, we see that from here on it suffices to consider only anomalous cases giving two solutions exactly one of which has $X=1$. 
Here is an example:
\begin{align*}
(1.)  \quad (a, b, c) &= (5,11,56):\\
&  5^5 + 11=56^2, 5 \cdot 11 + 1=56.  \\
\end{align*}
Using Observation A we obtain 
\begin{align*}
(2.) \quad (a, c, c^r+a^{s+q}) &= (5,56,15681):\\
&  5^6+ 56=15681, 5 \cdot 56^2 + 1=15681.
\end{align*}
For this example we have an application of Observation A with $(a,b,c,q,r,s,t)=(5,11,56,1,1,5,2)$.  Other such examples are given by 
$(a,b,c,q,r,s,t) = (2,1,3, \allowbreak  1,1,3,2)$, $(3,5,2,1,4,3,5)$, $(5,3,\allowbreak 2, \allowbreak  1, \allowbreak  4, \allowbreak 3,7)$, from which we derive three new $(d_1, d_2, c)$ with $N=2$:  
$$(3.) \quad  (d_1, d_2, c) = (2,3,19), (3,2,97), (5,2,641).$$

Aside from (1.), (2.),  and (3.) just listed and the Pillai-related double solutions, the only remaining anomalous double solutions with $n=N=2$ that we are aware of are: 
\begin{align*}
(d_1, d_2, c) &= (2,11,3):\\
&  2^4+11=3^3, 2 \cdot 11^2+1=3^5,
\end{align*}
\begin{align*}
(d_1, d_2, c) &= (8,35,99):\\
&  8^2+35=99, 8 \cdot 35^2 + 1=99^2
\end{align*}
(we can clearly replace $d_1=8$ by $d_1=2$ as well in this example).  
\begin{align*}
(d_1, d_2, c) &= (10,41,411):\\
&  10^5 + 41^3=411^2, 10 \cdot 41 + 1=411.
\end{align*}

\bigskip

\noindent
{\bf Appendix 2: Cases where $d_1$ and $d_2$ are both prime}

\bigskip

We now consider the case $n=2$ where we require the $d_1$ and $d_2$ to be prime or prime powers.  

Several of the anomalous solutions listed in Appendix 1 have $d_1$ and $d_2$ prime.  When $2 \notdiv c$, solutions in the listed infinite families with $d_1$ and $d_2$ both primes or prime powers must have $\{ d_1 , d_2 \} = \{ 2^r, J \}$ where $r>0$ and $J$ is a Mersenne or Fermat prime or $J = 9$, except for solutions in the first infinite family listed in Appendix 1 with $d_1 = k = 2^r$, $r>1$, $m>2$.  

We examine the latter case ($d_1= k = 2^r$, $r>1$, and $d_2 = \frac{2^{mr} -1}{2^r - 1}$, $m>2$, with $d_2$ a prime or prime power).  If $mr = 6$, then, since we are taking $r>1$ and $m>2$, we must have $m=3$, $r=2$, which gives $d_1 = 4$ and $d_2 = 21$, which is not a prime or prime power.  So we can assume $mr \ne 6$, and apply Theorem V of an old paper of Birkhoff and Vandiver [{\it Ann.  Math.}, 2nd series, {\bf 5}, no. 4, July 1904, pp. 173--180] to show that $mr$ can have no divisors less than itself which do not also divide $r$.  This requires $m=p$, $r=p^t$ where $p$ is prime and $t \ge 1$.  We know of four instances when $(2^{p^{t+1}}-1)/(2^{p^t}-1)$ is prime:   $p=3$, $t=1$;  $p=3$, $t=2$;  $p=7$, $t=1$;  $p=59$, $t=1$.   We have calculated that no other $(2^{p^{t+1}}-1)/(2^{p^t}-1)$ is prime for primes $p \le 37$ and $d_2 < 10^{30000}$, also that $(2^{p^{2}}-1)/(2^{p}-1)$ is not prime for any prime $p \le 541$, except for the cases $p=3$, 7, or 59 already listed.


\begin{thebibliography}{1}

\bibitem{BS} 
F. Beukers and H.P. Schlickewei,
The equation $x + y = 1$ in finitely generated groups,
{\it Acta Arith.}
{\bf 78} 1996.
189---199




\bibitem{C}
Z.-F. Cao,
On the Diophantine equation $ax^2 + by^2 = p^z$,
{\it J. Harbin Inst. Tech.} 
{\bf 23}
1991.
 108--111



\bibitem{Cao}
Z.-F. Cao, 
A note on the Diophantine equation $a^x + b^y = c^z$,
{\it Acta Arith.}
{\bf 91}
1999.
 85--93
 


\bibitem{CCS} 
Z. F. Cao, C. I. Chu, and W. C. Shiu,
The exponential Diophantine equation $A X^2 + B Y^2 = \lambda k^Z$ and its applications. 
{\it Taiwanese J. Math.}
{\bf 12}
2008.
1015--1034
  


\bibitem{Cas}
 J. W. S. Cassels,
On the equation $a^x -b^y = 1$,
{\it Amer. J. Math.}
{\bf 75}
1953.
 159--162




\bibitem{DYL}
N. Deng, P. Yuan, and W. Luo,
Number of solutions to $k a^x + l b^y = c^z$,
{\it J. of Number Theory}
{\bf 187}
2018.
 250--263




\bibitem{E} 
J. H. Evertse,
On equations in S-units and the Thue-Mahler equation,
{\it Invent. Math.}
{\bf 75}
1984.
 561--584

  


\bibitem{H-K}
N. Hirata-Kohno,
S-unit equations and integer solutions to exponential Diophantine equations,
{\it Analytic Number Theory and surrounding Areas 2006}
{\it Kyoto RIMS Kokyuroku}
2006.
  92--97



\bibitem{Le} 
Maohua Le, 
An upper bound for the number of solutions of the exponential Diophantine equation $a^x+b^y = c^z$,
{\it Proc. Japan Acad. Ser. A, Mat. Sci.}
{\bf 75}
1999.
 90--91

  
\bibitem{Le2}
Maohua Le, 
Some exponential Diophantine equations I: the equation $D_1 x_2 - D_2 y_2 = \lambda k^z$, 
{\it Journal of Number Theory}
{\bf 55}
1995.
209--221.  



\bibitem{Lev}
W. J. Leveque,
On the equation $a^x -b^y = 1$,
{\it Amer. J. Math.}
{\bf 74}
1952.
325--331

 

\bibitem{Lj}
W. Ljunggren, 
Noen Fetninger om ubestemte likninger av formen $(x^n - 1)/(x-1) = y^q$,
{\it Norsk. Mat. Tidsskr.}
{\bf 25}
1943.
17--20
   


\bibitem{Sc}
R. Scott, 
On the Equations $p^x-b^y = c$ and $a^x+b^y=c^z$, 
{\it J. Number Theory}
{\bf 44}
1993.
153--165




\bibitem{ScSt1}
R. Scott, R. Styer, 
On $p^x - q^y = c$ and related three term exponential Diophantine equations with prime bases, 
{\it J. of Number Theory}
{\bf 105}
2004.
212--234



\bibitem{ScSt6}
 R. Scott, R. Styer, 
Bennett's Pillai theorem with fractional bases and negative exponents allowed,
{\it J. de th\'eorie des nombres de Bordeaux}
{\bf 27}
2015.
289--307
 


\bibitem{ScSt7}
R. Scott, R. Styer, 
Number of solutions to $a^x + b^y = c^z$. 
{\it Publ. Math. Debrecen}
{\bf 88}
2016.
 131--138




\end{thebibliography}
\end{document}